\documentclass[a4paper,11pt,authoryear]{elsarticle}

\usepackage{amssymb}
\usepackage{amsmath}
\usepackage{amsthm}

\usepackage{multirow}
\usepackage{algorithm}
\usepackage{algpseudocode}
\usepackage{booktabs}
\usepackage{setspace}
\usepackage{threeparttable}
\usepackage{xurl} 
\usepackage[hidelinks]{hyperref}
\hypersetup{breaklinks=true}
\newtheorem{theorem}{Theorem}
\newtheorem{lemma}{Lemma}

\newtheorem{definition}{Definition}
\newdefinition{rmk}{Remark}
\newproof{pf}{Proof}
\newproof{pot}{Proof of Theorem \ref{thm}}

\journal{European Journal of Operational Research}

\begin{document}

\begin{frontmatter}



\author[1]{Junhao Wu}
\ead{junnhao_wu@163.com}

\author[2]{Shaoze Li}
\ead{shaoze_li@163.com}

\author[1]{Cheng Lu \corref{cor}}
\cortext[cor]{Corresponding author.}
\ead{lucheng1983@163.com}

\author[3]{Zhibin Deng}
\ead{zhibindeng@ucas.edu.cn}

\author[4]{Shu-Cherng Fang}
\ead{fang@ncsu.edu}

\title{Variable Aggregation-based Perspective Reformulation for Mixed-Integer Convex Optimization with Symmetry} 



\affiliation[1]{organization={School of Economics and Management, North China Electric Power University},
	city={Beijing},
	postcode={102206},
	country={China}}
	
\affiliation[2]{organization={Thayer School of Engineering, Dartmouth College},
	city={Hanover},
	postcode={NH 03755},
	country={USA}}

\affiliation[3]{organization={School of Economics and Management, University of Chinese Academy of Sciences},
	city={Beijing},
	postcode={100190},
	country={China}}

\affiliation[4]{organization={Department of Industrial and Systems Engineering, North Carolina State University},
	city={Raleigh},
	postcode={NC 27695-7906},
	country={USA}}
	
\begin{abstract}
This paper addresses the challenging issue of symmetry in mixed-integer convex optimization problems, which frequently arise in real-world applications such as the unit commitment problem. Although variable aggregation techniques have been employed to mitigate symmetry, their impact on tightening the corresponding continuous relaxation has not been thoroughly investigated. In this work, we propose a new formulation that integrates the perspective reformulation method into the variable aggregation framework, yielding a tighter continuous relaxation for mixed-integer convex optimization problems with symmetric structures. We prove that, in the presence of symmetry, the convex hull of the feasible region associated with each set of aggregated variables can be exactly characterized. These results demonstrate the effectiveness of the proposed reformulation and establish new theoretical foundations for achieving tightness in variable aggregation-based mixed-integer programming formulations.
\end{abstract}

%



\end{frontmatter}



\section{Introduction}\label{sec1}

Decision-making problems in real-life applications are often modeled as mixed-integer convex optimization problems. These problems require careful formulation to strike a balance between model accuracy and computational solvability. One significant challenge in developing an optimization model is managing the inherent symmetry of the problem \citep{Ostrowski2009,Sherali}. Symmetry in optimization can be defined by a bijection $\gamma:\mathbb{R}^n\rightarrow\mathbb{R}^n$, which maps a solution $\mathbf{x}$ to another solution $\gamma(\mathbf{x})$, while preserving both the objective value and feasibility. Symmetry can reduce the efficiency of branch-and-cut algorithms, as many symmetric solutions are redundantly enumerated, increasing computational effort \citep{clautiaux2025last}.

To address this issue, a variety of techniques have been proposed to break symmetry in mixed-integer programming (MIP) problems. These methods range from straightforward strategies, such as identifying a representative solution among symmetric alternatives and adding cuts to exclude non-representative solutions while keeping the representative solution feasible \citep{Anjos2017,Bendotti2021,Danielson,Friedman2007,Hojny2018,Hojny2020,Liberti2012,Liberti2011,Liberti2014}, to more advanced approaches like isomorphism pruning \citep{Margot,Margot2003,Margot2010} and orbital branching \citep{Ostrowski2011,Ostrowski2015}, which restrict the branching process to representative solutions only. Another approach to symmetry breaking involves problem reformulation, such as orbital shrinking \citep{Fischetti2012,Fischetti2017}, which projects solutions into a subspace known as the ``fixed space'', mapping each group of symmetric solutions to a common representative in the projected space. Variable aggregation-based reformulation, as an extension of orbital shrinking, has been proposed to solve the unit commitment problem \citep{Knueven2017}, demonstrating the utility of problem reformulation techniques on handling symmetry.


In addition to symmetry, another key challenge is deriving a tight continuous relaxation for mixed-integer programming problems, particularly those involving binary variables. One promising technique to improve the continuous relaxation is the perspective reformulation \citep{Gunluk2010}, which has been successfully applied to various problems involving binary variables \citep{Castro2014,Hijazi2012,zheng2021perspective}, including the unit commitment problem \citep{Bacci2023,Frangioni2009a}. Recent work by \cite{Bacci2023} introduced a perspective reformulation-based MIP formulation for the unit commitment problem, demonstrating that the continuous relaxation of this formulation is tight, meaning it exactly represents the convex hull of the feasible solutions for the single-unit commitment problem.

Building upon these advancements, this paper explores the integration of the perspective reformulation technique with variable aggregation to create a new formulation for a class of mixed-integer convex optimization problems. This hybrid approach aims to break symmetry while ensuring a tight continuous relaxation. Specifically, we propose a variable aggregation-based reformulation that incorporates perspective reformulation for a class of mixed-integer convex programming problems. We demonstrate that the continuous relaxation of this new formulation tightly represents the convex hull of the feasible solutions for each aggregated variable set.

The main contributions of this paper are as follows:

\begin{itemize}
	\item[$\bullet$] We propose a perspective reformulation that integrates the variable aggregation technique to handle symmetry while maintaining a tight continuous relaxation.
	\item[$\bullet$] We show that under some mild conditions, the convex hull of feasible solutions for aggregated variables can be exactly captured by the continuous relaxation of the proposed formulation.
	\item[$\bullet$] We establish new theoretical results that provide a foundation for tightness in variable aggregation-based perspective reformulation of a mixed-integer convex optimization problem.
\end{itemize}

The structure of the paper is as follows: Section \ref{sec2} introduces key concepts and problem formulations. Section \ref{sec3} presents the proposed variable aggregation-based perspective reformulation. Section \ref{sec4} analyzes the tightness of the continuous relaxation. Sections \ref{sec5} and \ref{sec6} apply the theoretical results to the unit commitment problem and general mixed-integer convex separable programming problems, respectively. Section \ref{sec7} presents computational results to validate our theoretical findings,
and Section \ref{sec8} concludes the paper.

In the remaining parts of this paper, we always use bold letters to represent matrices and vectors. For a vector $\mathbf{x}$, we use $x^j$ to represent its $j$-th entry. For a set of vectors $\mathbf{x}_t$ indexed by $t$, we use $x_t^j$ to represent the $j$-th entry of $\mathbf{x}_t$. To avoid confusion, we always use superscripts to index the entry of a vector or a matrix, and use subscripts to index a sequence of vectors or matrices. We use $\mathbb{R}_{+}$ to denote the set of nonnegative real numbers, and $\mathbb{Z}_{+}$ the set of nonnegative integers. For a set of vectors $\mathbf{x}_s~(s\in\mathcal{K})$, we use $\mathbf{x}:=\{\mathbf{x}_s\}_{s\in\mathcal{K}}$ to represent that $\mathbf{x}$ is defined by concatenating the vectors in $\{\mathbf{x}_s\}_{s\in\mathcal{K}}$. We use $\textrm{conv}(\cdot)$ to represent the convex hull of a set.
\section{Problem formulation}\label{sec2}

In this section, we introduce a general class of mixed-integer convex optimization problems. For easy description, we introduce three definitions first, which will be used throughout the paper.

\begin{definition}\label{def1}
	Let $p$ be a nonnegative integer and $\mathcal{F}\subseteq\mathbb{R}^d$ be a nonempty set, then
	$$p\otimes \mathcal{F}:=\left\{\sum_{i=1}^p\mathbf{x}_{i}\mid \mathbf{x}_{1},\ldots,\mathbf{x}_p\in \mathcal{F}\right\}$$ if $p>0$, and $p\otimes \mathcal{F}:=\{\mathbf{0}\}$ if $p=0$.
\end{definition}

\begin{definition}\label{def2}
	Let $r \in \mathbb{R}_{+}$ be a nonnegative real number and $\mathcal{F}\subseteq\mathbb{R}^d$ be a nonempty set, then $$r \odot \mathcal{F}:=\{r\mathbf{x}\mid \mathbf{x} \in \mathcal{F}\}$$ if $r>0$, and $r\odot \mathcal{F}:=\{\mathbf{0}\}$ if $r=0$.
\end{definition}

It is straightforward to check that ({\romannumeral1}) $(r_1r_2)\odot \mathcal{F}=r_1\odot(r_2\odot \mathcal{F})$ for any $r_1,r_2 \in \mathbb{R}_{+}$; and ({\romannumeral2}) the set $\{ (\mathbf{x},r)\mid r \in \mathbb{R}_{+} ~and~ \mathbf{x}\in r\odot \mathcal{F}\}$ is a closed set in $\mathbb{R}^{d+1}$ when $\mathcal{F}$ is a bounded and closed subset of $\mathbb{R}^{d}$.

\begin{definition}\label{def3}
	Let $f(\cdot)$ be a real-valued closed convex function with a closed convex effective domain $\mathcal{F}\subseteq\mathbb{R}^n$,  then its perspective function is
	\begin{equation}
		\begin{aligned}
			F(\mathbf{x},t)=\left\{\begin{array}{@{}lll}
				t f(\mathbf{x}/t),~&\textrm{if}~t>0\textrm{ and }\mathbf{x}/t\in\mathcal{F},\\
				0,~&\textrm{if}~t=0\textrm{ and }\mathbf{x}=\mathbf{0},\\
				+\infty,~&\textrm{otherwise}.
			\end{array}\right.
		\end{aligned}
	\end{equation}
\end{definition}
From \citep[Theorem 8.2 and Corollary 8.5.1]{Rockafellar}, we know that the closure of the perspective function $F(\mathbf{x},t)$ is
\begin{equation}
	\begin{aligned}
		\textrm{cl}F(\mathbf{x},t)=\left\{\begin{array}{@{}lll}
			t f(\mathbf{x}/t),~&\textrm{if}~t>0\textrm{ and }\mathbf{x}/t\in\mathcal{F}, \\
			f^{\infty}(\mathbf{x}),~&\textrm{if}~t=0\textrm{ and }\mathbf{x}\in \mathcal{F}^\infty, \\
			+\infty,~&\textrm{otherwise},
		\end{array}\right.
	\end{aligned}
\end{equation}
where $f^{\infty}(\mathbf{x})$ denotes the recession function of $f(\mathbf{x})$ and $\mathcal{F}^\infty$ denotes the recession cone of $\mathcal{F}$. One may refer to \citep[Section 8]{Rockafellar} for detailed discussions on the definitions and properties of the recession function and the recession cone. We remark that $f^{\infty}(\mathbf{0})=0$. Furthermore, if $\mathcal{F}$ is bounded, then $\mathcal{F}^\infty=\{\mathbf{0}\}$ and $\textrm{cl}F(\mathbf{x},t)=F(\mathbf{x},t)$.

We now introduce our problem formulation. Let $\mathcal{I}$ and $\mathcal{L}$ be finite index sets. For each $i \in \mathcal{I}$, we associate a $d$-dimensional real vector $\mathbf{x}_{i}=(x_{i}^1,x_{i}^2,\ldots,x_{i}^d)^\top\in \mathbb{R}^d$ and a $k$-dimensional binary vector $\mathbf{y}_{i}=(y_{i}^1,y_{i}^2,\ldots,y_{i}^k)^\top \\ \in \{0,1\}^k$.
We consider the following general mixed-integer convex optimization problem:
\begin{equation}\label{P-Main}
	\textrm{(MICOP)}~\left\{
	\begin{aligned}
		\min ~& \sum_{i\in \mathcal{I}}f_i(\mathbf{x}_i,\mathbf{y}_i) \\
		\mbox{s.t.}~&\sum_{i \in \mathcal{I}}g_{il}(\mathbf{x}_i,\mathbf{y}_i)\leq b_l,~l\in\mathcal{L}, \\
		&(\mathbf{x}_i,\mathbf{y}_i)\in\Omega_i,~i\in\mathcal{I},
	\end{aligned}
	\right.
\end{equation}
where for each $i\in\mathcal{I}$ and $l\in\mathcal{L}$, $f_i(\cdot)$ and $g_{il}(\cdot)$ are closed convex functions defined on $\mathbb{R}^{d+k}$, the feasible set $\Omega_i\subseteq \mathbb{R}^{d+k}$ is a set with convex constraints and binary constraints on the components of $\mathbf{y}_i$.

Furthermore, we specify the structure of the feasible set $\Omega_i$. In the class of problems under consideration, $\Omega_i$ is assumed to be bounded and has the following closed-form representation:
\begin{equation}\label{eq4}
	\Omega_i=\left\{ (\mathbf{x},\mathbf{y})
	\in \mathbb{R}^{d+k} 
	\left|\,\begin{aligned}
		&h_{il}(\mathbf{x}_s,y^s)\leq d_{l},~l\in\mathcal{H}_{is},~s\in\mathcal{K}\\
		&\mathbf{x}=\{\mathbf{x}_s\}_{s\in\mathcal{K}},~\mathbf{A}_i\mathbf{y} \leq \mathbf{b}_i,~\mathbf{y}\in \{0,1\}^k
	\end{aligned}\right.\right\},
\end{equation}
where $\mathcal{K}:=\{1,\ldots,k\}$, the vector $\mathbf{x}$ is partitioned into subvectors $\{\mathbf{x}_s\}_{s\in\mathcal{K}}$, for each $s\in\mathcal{K}$ there is a finite index set $\mathcal{H}_{is}$, such that for any $l\in \mathcal{H}_{is}$, $h_{il}(\cdot)$ is a closed convex function defined on variables $(\mathbf{x}_s,y^s)$, $\mathbf{A}_i$ is a totally unimodular matrix, and $\mathbf{b}_i$ is an integer vector. Without loss of generality, we always assume that 
$\{\mathbf{y}\in \mathbb{R}^{k}\,\mid\,\mathbf{A}_i\mathbf{y}\leq \mathbf{b}_i\}\subseteq[0,1]^k.$ 
We caution the reader that for the definition of $\Omega_i$, the generic variables $(\mathbf{x}, \mathbf{y})$ and their components $(\mathbf{x}_s, y^s)$ are used interchangeably with the indexed versions $(\mathbf{x}_i, \mathbf{y}_i)$ and $(\mathbf{x}_{is}, y_i^s)$ throughout this paper.

At the first glance, problem (MICOP) is very special. Actually, the formulation (MICOP) is initially motivated from the dynamic programming-based formulation of the unit commitment problem proposed in \citep{Bacci2023}. We intended to integrating the perspective reformulation method into the variable aggregation framework for the unit commitment problem. Subsequently, we found that the derived results can be generalized to some other problems such as the linear cover problem \citep{Agnetis2009,Agnetis2012}, and some general separable mixed-integer convex optimization problems \citep{Bretthauer2002,Edirisinghe2019,Frangioni2011,Schoot2022}. Thus, (MICOP) can be seen as a generalization of the dynamic programming-based formulation of the unit commitment problem proposed in \citep{Bacci2023}, while provides a unified formulation for the aforementioned problems.

For (MICOP), we say that the two indices $i$ and $j$ in $\mathcal{I}$ are equivalent if they share the same parameters in the objective function and constraints (in other word, $\Omega_i=\Omega_j$, $f_{i}=f_j$, and $g_{il}=g_{jl}$ for all $l\in\mathcal{L}$). We partition the set $\mathcal{I}$ into $T$ mutually disjoint subsets, i.e., $\mathcal{I}=\bigcup_{t=1}^T \mathcal{I}_{t}$, where $T$ represents the number of different equivalent classes and each set $\mathcal{I}_t$ ($t=1,\ldots,T$) represents a set of indices (machines) in the same equivalent class. Let $\mathcal{T}:=\{1,\ldots,T\}$. For each equivalent class $t\in\mathcal{T}$, denote the identical sets, objective functions and constraint functions, shared by all $i\in\mathcal{I}_t$, as $\Omega_t$, $f_{t}(\cdot)$ and $g_{tl}(\cdot)$ for $l\in\mathcal{L}$, respectively, where $\Omega_t$ is defined as 
follows:
\begin{equation}\label{eq5}
	\Omega_t=\left\{ (\mathbf{x},\mathbf{y})
	\in \mathbb{R}^{d+k} 
	\left|\,\begin{aligned}
		&h_{tl}(\mathbf{x}_s,y^s)\leq d_l,~l\in\mathcal{H}_{ts},~s\in\mathcal{K}\\
		&\mathbf{x}=\{\mathbf{x}_s\}_{s\in\mathcal{K}},~\mathbf{A}_{t}\mathbf{y} \leq \mathbf{b}_t,~\mathbf{y}\in \{0,1\}^k
	\end{aligned}\right.\right\}.
\end{equation}
With the above notations, (MICOP) can be reformulated as follows:
\begin{equation}\label{P2}
	\textrm{(SMICOP)}~\left\{
	\begin{aligned}
		\min ~& \sum_{t\in\mathcal{T}}\sum_{i\in \mathcal{I}_{t}}f_t(\mathbf{x}_i,\mathbf{y}_i) \\
		\mbox{s.t.}~&\sum_{t\in\mathcal{T}}\sum_{i \in \mathcal{I}_{t}}g_{tl}(\mathbf{x}_i,\mathbf{y}_i)\leq b_l,~l\in\mathcal{L}, \\
		&(\mathbf{x}_i,\mathbf{y}_i)\in \Omega_t,~i\in \mathcal{I}_t,~t\in\mathcal{T}.
	\end{aligned}
	\right.
\end{equation}
The main objective of this paper is to handle the symmetry lies within (SMICOP).

\section{Integration of perspective reformulation with variable aggregation}\label{sec3}
In this section, by incorporating the method of perspective reformulation, we derive a variable aggregation technique to handle the symmetry lies within (SMICOP). Without loss of generality, we assume that for all $t\in \mathcal{T}$ and $l\in\mathcal{L}$, the functions $f_t(\cdot)$ and $g_{tl}(\cdot)$ are linear functions. In the remaining parts of this paper, we will focus on the following problem:
\begin{equation}
	\textrm{(P0)}~\left\{
	\begin{aligned}
		\min ~& \sum_{t\in\mathcal{T}}\sum_{i\in \mathcal{I}_{t}}\mathbf{c}_t^\top\mathbf{x}_i \\
		\mbox{s.t.}~&\sum_{t\in\mathcal{T}}\sum_{i \in \mathcal{I}_{t}}\mathbf{a}_{tl}^\top \mathbf{x}_i\leq b_l,~l\in\mathcal{L}, \\
		&(\mathbf{x}_i,\mathbf{y}_i)\in \Omega_t,~i\in \mathcal{I}_t,~t\in\mathcal{T},
	\end{aligned}
	\right.
\end{equation}
where $\mathbf{c}_t\in\mathbb{R}^d$ and $\mathbf{a}_{tl}\in\mathbb{R}^d$ for each $t\in\mathcal{T}$ and $l\in\mathcal{L}$.
In (P0), all the nonlinear constraints and integral constraints are packed into the constraint $(\mathbf{x}_i,\mathbf{y}_i)\in\Omega_t$. The objective function and the remaining constraints are all linear with respective to the continuous variables $\mathbf{x}_i$ ($i\in\mathcal{I}$). Actually, (P0) includes (SMICOP) as a special case. A proof for deriving (SMICOP) as a special case of (P0) is presented in Appendix A of Supplementary Materials.


To handle symmetry in problem (P0), we define the aggregated variables for each $t\in \mathcal{T}$, namely,
\begin{equation*}
	\begin{aligned}
		\mathbf{X}_{t}=\sum_{i \in \mathcal{I}_t}\mathbf{x}_i,~\mathbf{Y}_{t}=\sum_{i \in \mathcal{I}_t}\mathbf{y}_i.
	\end{aligned}
\end{equation*}
Using the aggregated variables, (P0) is reformulated as follows:
\begin{equation}\label{P4}
	\textrm{(P1)}~\left\{
	\begin{aligned}
		\min ~& \sum_{t\in\mathcal{T}}\mathbf{c}_t^\top\mathbf{X}_t \\
		\mbox{s.t.}~&\sum_{t\in\mathcal{T}}\mathbf{a}_{tl}^\top\mathbf{X}_t\leq b_l,~l\in\mathcal{L}, \\
		&(\mathbf{X}_t,\mathbf{Y}_t)\in N_t\otimes \Omega_t,~t\in\mathcal{T},
	\end{aligned}
	\right.
\end{equation}
where $N_t=|\mathcal{I}_t|$. Actually, (P1) is an abstract formulation rather than a detailed formulation. 
As discussed in \citep{Knueven2017}, if (P0) is a mixed-integer linear optimization problem, where the functions $h_{tl}(\cdot)~ (\forall l \in \mathcal{H}_{ts},~ s \in \mathcal{K})$ are all linear functions defined on the variables $(\mathbf{x}_s, y^s)$, a detailed formulation of (P1) can be derived directly using variable aggregation techniques to break the symmetry in problem (P0). However, for the general mixed-integer convex optimization problems considered in this work, particularly those with feasible sets of the form \eqref{eq4}, existing variable aggregation techniques are insufficient to derive the closed-form expressions for $N_t\otimes \Omega_t$. As a result, a detailed formulation of (P1) cannot be obtained to break the symmetry in (P0). Addressing this challenge is a central contribution of this paper. In the remaining parts of this section, we discuss how to derive closed-form expressions for $(\mathbf{X}_t,\mathbf{Y}_t)\in N_t\otimes \Omega_t$. Before deriving the closed-form expressions for (P1), 
we first present some general theoretical results for a general nonempty convex set $\mathcal{F}\subseteq\mathbb{R}^{d}$.

\begin{lemma}\label{lem1}
	If $\mathcal{F}\subseteq\mathbb{R}^{d}$ is a nonempty convex set and $r\in\mathbb{Z}_+$, then $r\otimes \mathcal{F}=r\odot \mathcal{F}$.
\end{lemma}
\begin{proof}
Please see Appendix B in Supplementary Materials for the proof.
\end{proof}

Now, we further assume that the set $\mathcal{F}$ is a nonempty bounded closed convex set defined as
\begin{equation}\label{eq8}
	\mathcal{F}=\{\mathbf{x}\in \mathbb{R}^{d}\,\mid\,h_{l}(\mathbf{x})\leq d_l,~l\in\mathcal{H}\},
\end{equation}
where $\mathcal{H}$ denotes an index set and, for each $l\in\mathcal{H}$, $d_l\in \mathbb{R}$ and the function $h_l(\cdot)$ is closed and convex. We then have the next result.

\begin{theorem}\label{thm1}
	Let $\mathcal{F}$ be a nonempty bounded closed convex set defined as in \eqref{eq8} and $r\in\mathbb{R}_+$, then the set $r\odot \mathcal{F}$ can be represented as
	\begin{equation}\label{eq9}
		r \odot \mathcal{F}=\{\mathbf{X}\in\mathbb{R}^d\,\mid\,H_l(\mathbf{X},r)\leq rd_l,~l\in\mathcal{H}\},
	\end{equation}
	where $H_l(\cdot,\cdot)$ denotes the perspective function of $h_l(\cdot)$ for $l\in\mathcal{H}$.
\end{theorem}
\begin{proof}
	Please see Appendix C in Supplementary Materials for the proof.
\end{proof}

Returning to (P1), we consider the closed-form expressions of the constraints $(\mathbf{X}_t,\mathbf{Y}_t)\in N_t\otimes \Omega_t$ for $t\in\mathcal{T}$. 
For simplicity of notations, we temporarily ignore the subscript $t$ of $\Omega_t$ and consider the set $\Omega\in \mathbb{R}^{d+k}$ defined by the following general formulation:
\begin{equation}\label{eq12}
	\Omega=\left\{ (\mathbf{x},\mathbf{y})
	\in \mathbb{R}^{d+k} 
	\left|\,\begin{aligned}
		&h_{l}(\mathbf{x}_s,y^s)\leq d_l,~l\in\mathcal{H}_{s},~s\in\mathcal{K}\\
		&\mathbf{x}=\{\mathbf{x}_s\}_{s\in\mathcal{K}},~\mathbf{A}\mathbf{y} \leq \mathbf{b},~\mathbf{y}\in \{0,1\}^k
	\end{aligned}\right.\right\}.
\end{equation}
where for each $s\in\mathcal{K}$ and $l\in\mathcal{H}_s$, the function $h_l(\cdot)$ is closed and convex, and $d_l\in \mathbb{R}$.
For each $l\in\mathcal{H}_s$, let $\hat{h}_l(\mathbf{x}_s):=h_l(\mathbf{x}_s,1)$ and $\bar{h}_l(\mathbf{x}_s):=h_l(\mathbf{x}_s,0)$. Note that $\hat{h}_l(\cdot)$ and $\bar{h}_l(\cdot)$ are all convex, since $h_l(\cdot,\cdot)$ is convex. We also define the following two convex sets:
\begin{equation}
	\begin{aligned}
		\Lambda_s=\{ \mathbf{x}_s\,|\, \hat{h}_l(\mathbf{x}_s)\leq d_l,~l\in\mathcal{H}_s\},
	~\textrm{and}~
		\Gamma_s=\{ \mathbf{x}_s\,|\, \bar{h}_l(\mathbf{x}_s)\leq d_l,~l\in\mathcal{H}_s\}.
	\end{aligned}
\end{equation}
Using these notations, the set $\Omega$ can be reformulated as
\begin{equation}\label{eq14}
	\Omega =\left\{ (\mathbf{x},\mathbf{y})\left|\begin{aligned}
		&\mathbf{x}_s\in y^s\otimes \Lambda_s+(1-y^s)\otimes \Gamma_s, s\in \mathcal{K}\\
		&\mathbf{x}=\{\mathbf{x}_s\}_{s\in\mathcal{K}},~\mathbf{A}\mathbf{y}\leq \mathbf{b}, ~\mathbf{y}\in \{0,1\}^k
	\end{aligned}\right.\right\}.
\end{equation}
For each $l\in\mathcal{H}_s$, let $\hat{H}_{l}$ and $\bar{H}_{l}$ be the perspective function of $\hat{h}_{l}$ and $\bar{h}_{l}$, respectively. Then, following (13), Lemma \ref{lem1} and Theorem \ref{thm1}, the set $\Omega$ can be reformulated as 

\begin{equation}\label{eq15}
	\Omega^{\textrm{per}}:=\left\{(\mathbf{x},\mathbf{y})\left|\begin{aligned}
		&\hat{H}_{l}(\mathbf{w}_s,y^s)\leq y^s d_l, ~ l\in \mathcal{H}_s,~s\in\mathcal{K}\\
		&\bar{H}_{l}(\mathbf{z}_s,1-y^s)\leq (1-y^s)d_l, ~ l\in \mathcal{H}_s,~s\in\mathcal{K}\\
		&\mathbf{x}_s=\mathbf{w}_s+\mathbf{z}_s,~s\in\mathcal{K}\\
		&\mathbf{x}=\{\mathbf{x}_s\}_{s\in\mathcal{K}},~\mathbf{A}\mathbf{y}\leq \mathbf{b}, ~\mathbf{y}\in \{0,1\}^k
	\end{aligned}\right.\right\}.
\end{equation}
The formulation of $\Omega^{\textrm{per}}$ can be seen as the perspective reformulation of $\Omega$ with the formulation given in \eqref{eq12}. 

Then, we have the next result.
\begin{theorem}\label{thm2}
	Let $\Omega$ be a nonempty set formulated as in \eqref{eq14}. Then for any $r\in\mathbb{Z}_+$, we have
	\begin{equation}\label{eq16}
		\begin{aligned}
			r\otimes\Omega=\Omega^{\textrm{agg}}:=\left\{ (\mathbf{X},\mathbf{Y}) \left|\begin{aligned}
				&\mathbf{X}_s\in Y^s\odot \Lambda_s+(r-Y^s)\odot \Gamma_s, ~s \in \mathcal{K}\\
				&\mathbf{X}=\{\mathbf{X}_{s}\}_{s\in \mathcal{K} },~\mathbf{A} \mathbf{Y}\leq r\mathbf{b}, ~ \mathbf{Y}\in \{0,1,\ldots,r\}^{k}
			\end{aligned}\right.\right\}.
		\end{aligned}
	\end{equation}
\end{theorem}
\begin{proof}
	Please see Appendix D in Supplementary Materials for the proof.
\end{proof}

Based on \eqref{eq16}, the closed-form formulation of $\Omega^{\textrm{agg}}$ can be further expressed as
\begin{equation}\label{eq18}
	\begin{aligned}
		\Omega^{\textrm{agg}}=
		\left\{ (\mathbf{X},\mathbf{Y})  \left|\begin{aligned}
			&\hat{H}_{l}(\mathbf{W}_s,Y^s)\leq Y^s d_l, ~ l\in \mathcal{H}_s,~s\in\mathcal{K}\\
			&\bar{H}_{l}(\mathbf{Z}_s,r-Y^s)\leq (r-Y^s)d_l, ~ l\in \mathcal{H}_s,~s\in\mathcal{K}\\
			&\mathbf{X}_s=\mathbf{W}_s+\mathbf{Z}_s,~s\in\mathcal{K}\\
			&\mathbf{X}=\{\mathbf{X}_{s}\}_{s\in \mathcal{K} },~\mathbf{A} \mathbf{Y}\leq r\mathbf{b}, ~ \mathbf{Y}\in \{0,1,\ldots,r\}^{k}
		\end{aligned}\right.\right\}.
	\end{aligned}
\end{equation}

Now we establish a closed-form formulation of (P1). 
For each $l\in\mathcal{H}_{ts}$, let $\hat{H}_{tl}$ and $\bar{H}_{tl}$ be the perspective function of $\hat{h}_{tl}$ and $\bar{h}_{tl}$, respectively.
Based on Theorem 2, the closed-form formulation of (P1) can be derived as 
\begin{equation}
	\textrm{(P-agg)}~\left\{
	\begin{aligned}
		\min ~& \sum_{t\in\mathcal{T}}\mathbf{c}_{t}^\top\mathbf{X}_{t}\\
		\mbox{s.t.}~&\sum_{t\in\mathcal{T}} \mathbf{a}_{tl}^\top \mathbf{X}_{t}\leq b_l, ~ l\in\mathcal{L},\\
		&\mathbf{X}_t=\{\mathbf{X}_{ts}\}_{s\in\mathcal{K}},~t\in\mathcal{T},\\
		&\mathbf{X}_{ts}=\mathbf{W}_{ts}+\mathbf{Z}_{ts}, ~ s\in \mathcal{K},~t\in\mathcal{T},\\
		&\hat{H}_{tl}(\mathbf{W}_{ts},Y_{t}^s)\leq Y_t^s d_{tl}, ~s\in \mathcal{K},~l\in\mathcal{H}_{ts}, ~t\in\mathcal{T},\\
		&\bar{H}_{tl}(\mathbf{Z}_{ts},N_t-Y_{t}^s)\leq (N_t-Y_{t}^s) d_{tl},~s\in \mathcal{K},~l\in\mathcal{H}_{ts},~t\in\mathcal{T},\\
		&\mathbf{A}_t \mathbf{Y}_{t}\leq N_t\mathbf{b}_t, ~\mathbf{Y}_t\in \{0,1,\ldots,N_t\}^k, ~ t\in\mathcal{T}.
	\end{aligned}
	\right.
\end{equation}
We call (P-agg) the aggregated reformulation of (P0). The technique of deriving (P-agg) integrates the methods of perspective reformulation and variable aggregation, thus is called a variable aggregation-based perspective reformulation. Such approach can break the symmetry in (P0) effectively.

We remark that if $h_{tl}(\cdot)$ is a convex quadratic function, then both 
\[
\hat{H}_{tl}(\mathbf{W}_{ts}, Y_t^s) \le Y_t^s d_{tl} \quad \text{and} \quad \bar{H}_{tl}(\mathbf{Z}_{ts}, N_t - Y_t^s) \le (N_t - Y_t^s)d_{tl}
\] are second-order cone representable \citep{Frangioni2009,Gunluk2010}. 
Thus, if the original problem (P0) is a mixed-integer convex quadratic optimization problem, then (P-agg) can be formulated as a mixed-integer second-order cone optimization problem.

\section{Tightness of the proposed reformulation}\label{sec4}

In the literature, it is known that the perspective reformulation technique can improve the tightness of the continuous relaxation of a mixed-integer convex optimization problem significantly. It is of interest to investigate whether the variable aggregation-based perspective reformulation technique has the same effect.
In this section, we study the tightness of the continuous relaxation of the proposed reformulation (P-agg). We first prove that under some mild assumptions, the convex hull of $N_t\otimes\Omega_t$ can be defined by the continuous relaxation of the formulation given in \eqref{eq18}. Then, we further compare the tightness of the continuous relaxations of various formulations of (P0).

\subsection{Convex hull of the aggregated formulation}

In this section, we derive the convex hull of the set $N_t\otimes\Omega_t$. For simplicity, we ignore the subscript $t$ and denote $\Omega_t$ as $\Omega$. Consider the set $\Omega^{\textrm{agg}}$ defined in \eqref{eq18}, whose continuous relaxation is denoted by $\tilde{\Omega}^{\textrm{agg}}$, defined as

\begin{equation}
	\begin{aligned}
		\tilde{\Omega}^{\textrm{agg}}=
		\left\{ (\mathbf{X},\mathbf{Y})  \left|\begin{aligned}
			&\hat{H}_{l}(\mathbf{W}_s,Y^s)\leq Y^s d_l, ~ l\in \mathcal{H}_s,~s\in\mathcal{K}\\
			&\bar{H}_{l}(\mathbf{Z}_s,r-Y^s)\leq (r-Y^s)d_l, ~ l\in \mathcal{H}_s,~s\in\mathcal{K}\\
			&\mathbf{X}_s=\mathbf{W}_s+\mathbf{Z}_s,~s\in\mathcal{K}\\
			&\mathbf{X}=\{\mathbf{X}_{s}\}_{s\in \mathcal{K} },~\mathbf{A} \mathbf{Y}\leq r\mathbf{b}
		\end{aligned}\right.\right\}.
	\end{aligned}
\end{equation}
One main theoretical result in this section is to show that $\textrm{conv}(r\otimes\Omega)=\tilde{\Omega}^{\textrm{agg}}$.

We first consider the set $\Omega^{\textrm{per}}$, whose continuous relaxation is denoted by $\tilde{\Omega}^{\textrm{per}}$, defined as 
\begin{equation}
	\tilde{\Omega}^{\textrm{per}} =
	\left\{(\mathbf{x},\mathbf{y})\left|\begin{aligned}
		&\hat{H}_{l}(\mathbf{w}_s,y^s)\leq y^s d_l,~ l\in \mathcal{H}_s,~s\in\mathcal{K}\\
		&\bar{H}_{l}(\mathbf{z}_s,1-y^s)\leq (1-y^s)d_l,~ l\in \mathcal{H}_s,~s\in\mathcal{K}\\
		&\mathbf{x}_s=\mathbf{w}_s+\mathbf{z}_s,~s\in\mathcal{K}\\
		&\mathbf{x}=\{\mathbf{x}_s\}_{s\in\mathcal{K}},~\mathbf{A}\mathbf{y}\leq \mathbf{b}
	\end{aligned}\right.\right\}.
\end{equation}
We assume that the following condition holds:

\textbf{Assumption 1:} $\tilde{\Omega}^{\textrm{per}}$ contains a relative interior-point solution.

Then we have the following two lemmas:

\begin{lemma}\label{lem2}
	Under Assumption 1, we have $\textrm{conv}(\Omega)=\tilde{\Omega}^{\textrm{per}}$.
\end{lemma}
\begin{proof}
	Please see Appendix E in Supplementary Materials for the proof.
\end{proof}

\begin{lemma}\label{lemx}
	Let $\mathcal{F}\subseteq\mathbb{R}^{d}$ be a nonempty bounded closed set and $r$ be a nonnegative integer, then we have $\textrm{conv}(r\otimes \mathcal{F})=r\otimes \textrm{conv}(\mathcal{F})$.
\end{lemma}
\begin{proof}
	Please see Appendix F in Supplementary Materials for the proof.
\end{proof}

Now, we present our major theoretical result showing that the continuous relaxation $\tilde{\Omega}^{\textrm{agg}}$ is tight.
\begin{theorem}\label{thm3}
	Under Assumption 1, for any $r\in\mathbb{Z}_+$, we have
	$\tilde{\Omega}^{\textrm{agg}}=\textrm{conv}(r\otimes\Omega)=r\odot\tilde{\Omega}^{\textrm{per}}.$
\end{theorem}
\begin{proof}
	Please see Appendix G in Supplementary Materials for the proof.
\end{proof}

It is well-known that with the aid of perspective reformulation, the convex hull of several mixed-integer convex sets can be derived (as discussed in Section~\ref{sec1}). However, in the existing literature, the perspective reformulation has never been integrated with the variable aggregation technique. In this work, we do so and study its impact to the continuous relaxation of the aggregated formulation for the first time.

\subsection{Tightness results for various formulations}
We further discuss the relationship between the continuous relaxations of various formulations. In the previous sections, we have introduced two formulations, the original formulation (P0) and the aggregated formulation (P-agg). Besides these two formulations, since the set $\Omega_t$ can be reformulated as $\Omega_t^{\textrm{per}}$ using \eqref{eq15}, we can derive the following perspective reformulation of (P0):
\begin{equation}
	\textrm{(P-per)}~\left\{
	\begin{aligned}
		\min & \sum_{t\in\mathcal{T}}\sum_{i\in \mathcal{I}_{t}}\mathbf{c}_t^\top\mathbf{x}_i \\
		\mbox{s.t.}&\sum_{t\in\mathcal{T}}\sum_{i \in \mathcal{I}_{t}}\mathbf{a}_{tl}^\top \mathbf{x}_i\leq b_l,~l\in\mathcal{L}, \\
		&\mathbf{x}_i=\{\mathbf{x}_{is}\}_{s\in\mathcal{K}},~i \in \mathcal{I}_{t},~t\in\mathcal{T},\\
		&\mathbf{x}_{is}=\mathbf{w}_{is}+\mathbf{z}_{is}, ~s\in \mathcal{K},~i \in \mathcal{I}_{t},~t\in\mathcal{T},\\
		&\hat{H}_{tl}(\mathbf{w}_{is},y_i^s)\leq y_i^s d_l, ~\bar{H}_{tl}(\mathbf{z}_{is},1-y_i^s)\leq (1-y_i^s)d_l,\\
		&\quad\quad\quad\quad\quad\quad\quad\quad\quad\quad~ s\in \mathcal{K},~l\in\mathcal{H}_{ts},~i \in \mathcal{I}_{t},~t\in\mathcal{T},\\
		&\mathbf{A}_t\mathbf{y}_i\leq \mathbf{b}_t, ~\mathbf{y}_i\in \{0,1\}^k,~i \in \mathcal{I}_{t},~t\in\mathcal{T}.
	\end{aligned}
	\right.
\end{equation}
The main purpose of this subsection is to compare the tightness of continuous relaxations of (P0), (P-per) and (P-agg). For (P0), the continuous relaxation is given as
\begin{equation}\label{cP0}
	\textrm{}~\left\{
	\begin{aligned}
		\min ~& \sum_{t\in\mathcal{T}}\sum_{i\in \mathcal{I}_{t}}\mathbf{c}_t^\top\mathbf{x}_i \\
		\mbox{s.t.}~&\sum_{t\in\mathcal{T}}\sum_{i \in \mathcal{I}_{t}}\mathbf{a}_{tl}^\top \mathbf{x}_i\leq b_l,~l\in\mathcal{L}, \\
		&(\mathbf{x}_i,\mathbf{y}_i)\in \mathcal{C}_t,~i\in \mathcal{I}_t,~t\in\mathcal{T},
	\end{aligned}
	\right.
\end{equation}
where $\mathcal{C}_t$ denotes the continuous relaxation of $\Omega_t$ in its original form defined in \eqref{eq12} without using perspective reformulation, expressed as
\begin{equation}\label{eq10}
	\mathcal{C}_t=\left\{ (\mathbf{x},\mathbf{y})
	\left|\,\begin{aligned}
		&h_{tl}(\mathbf{x}_s,y^s)\leq d_{tl}, ~l\in\mathcal{H}_{ts},~s\in\mathcal{K}\\
		&\mathbf{x}=\{\mathbf{x}_s\}_{s\in\mathcal{K}},~\mathbf{A}_t\mathbf{y} \leq \mathbf{b}_t
	\end{aligned}\right.\right\}.
\end{equation}
Using the notations defined in Section 4.1, the continuous relaxation of (P-per) can be formulated as
\begin{equation}\label{ppre}
	\textrm{}\left\{
	\begin{aligned}
		\min ~& \sum_{t\in\mathcal{T}}\sum_{i\in \mathcal{I}_{t}}\mathbf{c}_t^\top\mathbf{x}_i \\
		\mbox{s.t.}~&\sum_{t\in\mathcal{T}}\sum_{i \in \mathcal{I}_{t}}\mathbf{a}_{tl}^\top \mathbf{x}_i\leq b_l,~l\in\mathcal{L}, \\
		&(\mathbf{x}_i,\mathbf{y}_i)\in \tilde{\Omega}_t^{\textrm{per}},~i\in\mathcal{I}_t,~t\in \mathcal{T},
	\end{aligned}
	\right.
\end{equation}
and the continuous relaxation of (P-agg) can be formulated as
\begin{equation}\label{pare}
	\textrm{}~\left\{
	\begin{aligned}
		\min ~& \sum_{t\in\mathcal{T}}\mathbf{c}_{t}^\top\mathbf{X}_{t}\\
		\mbox{s.t.}~&\sum_{t\in\mathcal{T}} \mathbf{a}_{tl}^\top \mathbf{X}_{t}\leq b_l,~ l\in\mathcal{L},\\
		&(\mathbf{X}_t,\mathbf{Y}_t)\in  \tilde{\Omega}_t^{\textrm{agg}},~t\in \mathcal{T}.
	\end{aligned}
	\right.
\end{equation}

We use LB$_{(\cdot)}$ to represent the lower bound computed by the continuous relaxation of problem $(\cdot)$. The tightness of the three formulations (P0), (P-per) and (P-agg) is shown in the next result.

\begin{theorem}\label{thm4}
	For the tightness of the continuous relaxations of (P0), (P-per) and (P-agg), we have  $\textrm{LB}_{(\textrm{P-agg})}=\textrm{LB}_{(\textrm{P-per})}\geq  \textrm{LB}_{(\textrm{P0})}$.
\end{theorem}
\begin{proof}
	Please see Appendix H in Supplementary Materials for the proof.
\end{proof}

Our theoretical results indicate that by integrating the perspective reformulation with the variable aggregation technique to derive (P-agg), not only can the symmetry be broken, but also the tightness of continuous relaxation can be improved significantly.

\section{Application to the unit commitment problem}\label{sec5}
The unit commitment problem in power systems is to determine the operation states and generation power of a set of generating units $\mathcal{I}$ over a period of time $\mathcal{M} := \{1, \ldots, n\}$ to minimize the total operation cost. Several mixed-integer programming formulations for the problem have been proposed in the literature \citep{Knueven2020}. One of the most well-known formulation is the classical 3-bin formulation \citep{Rajan2005}. However, due to the ramp constraints, the classical 3-bin formulation fails to satisfy the structure of (P0), so that the variable aggregated reformulation is not applicable. Thus, we consider another formulation, known as the dynamic programming-based formulation, that was proposed in \cite{Bacci2023}. For completeness of this paper, we provide a sketch of the dynamic programming-based formulation. One may refer to \cite{Bacci2023} for more details.

In the dynamic programming-based formulation, a generating unit may have two operating states: on and off. In order to formulate a valid state sequence that satisfies the minimum up and down time constraints, a graph $G_i=\{\mathcal{N}_i,\mathcal{A}_i\}$ is introduced for each generating unit $i\in\mathcal{I}$. The node set $\mathcal{N}_i$ contains two types of nodes: $ON_{ij}$ and $OFF_{ij}$ for each $j\in \mathcal{M}$, plus two special nodes, source $S_i$ and sink $D_i$. The set $\mathcal{A}_i$ contains two types of arcs: One is the ON arc, $(OFF_{ih}, ON_{ir})$, representing that the generating unit is turned on in the time period $h$, and remains in the on state until the time period $r$. Another is the OFF arc, $(ON_{ih}, OFF_{ir})$, representing that the generating unit is turned off from time period $h+1$ to the time period $r-1$. The nodes $S_i$ and $D_i$ are also connected to (or connected by) some suitable nodes to define ON and OFF arcs.

Given the graph $G_i=\{\mathcal{N}_i,\mathcal{A}_i\}$, a valid state sequence of generating unit $i$ can be represented by a path from $S_i$ to $D_i$. Then, a vector $\mathbf{y}_{i}\in\{0,1\}^{k_i}$ is introduced, where $k_i=|\mathcal{A}_i|$, and each entry of $\mathbf{y}_{i}$ is associated with an arc in $\mathcal{A}_i$. Let $\mathcal{K}_i:=\{1,\ldots,k_i\}$. The element $s\in\mathcal{K}_i$ indexes both the entry of $\mathbf{y}_{i}$ and the associated arc. Then, a path from $S_i$ to $D_i$ can be represented by using the network flow formulation, $\mathbf{E}_{i}\mathbf{y}_{i}=\mathbf{\delta}_{i}$,
where $\mathbf{E}_{i}$ is the node-arcs incidence matrix of the graph $G_i$ and $\mathbf{\delta}_{i}\in \{0,1\}^{|\mathcal{N}_i|}$ with each entry being associated with a node, and all entries of $\mathbf{\delta}_{i}$ equaling zero, except the entries $\delta_{S_i}=-1$ and $\delta_{D_i}=1$ for the source and sink nodes, respectively. We remark that $\mathbf{E}_{i}$ is a totally unimodular matrix.

The set $\mathcal{K}_i$ is partitioned into two subsets $\mathcal{K}_i^+$ and $\mathcal{K}_i^-$, which represent the subsets of ON arcs and OFF arcs, respectively. Note that for an OFF arc, the generating unit is turned off and the power is fixed to zero. Hence, we only need to determine the power over the time periods associated with an ON arc. Let $s\in \mathcal{K}_i^+$ with associated time periods from $h$ to $r$. A vector $\mathbf{x}_{is}\in \mathbb{R}^{n}$ is introduced to represent the power output of the generating unit over the time periods from $h$ to $r$, while the power output being not in these time periods is fixed to zero. The set $\Lambda_{is}$ is defined to represent some physical limits on the power, including minimum and maximum power limits, start-up and shut-down limits and ramp limits. Moreover, $\mathbf{d}$ denotes the system load over the whole time horizon.

Finally, the total cost for each generating unit $i\in\mathcal{I}$ can be represented as $\mathbf{c}_{i}^\top\mathbf{y}_{i}+\sum_{s \in \mathcal{K}_i} f_{is}(\mathbf{x}_{is})$, where $\mathbf{c}_{i}$ represents the fixed operation cost, and $f_{is}(\mathbf{x}_{is})$ represents the variable operation cost which is usually a separable convex quadratic function. In detail, when $s\in\mathcal{K}_i^-$, the entry $c_{is}$ of $\mathbf{c}_{i}$ is equal to the start-up cost. On the other hand, when $s\in\mathcal{K}_i^+$, the entry $c_{is}$ is equal to $\tau_s\hat{c}_i$, where $\hat{c}_i$ represents the minimum running cost of generator $i$, and $\tau_s$ represents the length of the time interval associated with the ON arc. Clearly, for each $s \in \mathcal{K}_i^+$, we can move the nonlinear part of the total cost $f_{is}(\mathbf{x}_{is})$ into the constraints by introducing an auxiliary variable $z_{is}$, so that the set $\Lambda_{is}$ further contains the constraints $f_{is}(\mathbf{x}_{is})\leq z_{is}$ and $\underline{z}_{is}\leq z_{is}\leq \bar{z}_{is}$, where $\underline{z}_{is}=\min_{\mathbf{x}_{is}\in \Lambda_{is}}f_{is}(\mathbf{x}_{is})$ and $\bar{z}_{is}=\max_{\mathbf{x}_{is}\in \Lambda_{is}}f_{is}(\mathbf{x}_{is})$.

Based on the above descriptions, the unit commitment problem can be formulated as follows:
\begin{equation}\label{UC}
	\textrm{(UC)}~\left\{
	\begin{aligned}
		\min ~& \sum_{i\in \mathcal{I}}\left(\mathbf{c}_{i}^\top\mathbf{y}_{i}+\sum_{s \in \mathcal{K}_i^+}z_{is}\right)\\
		\mbox{s.t.}~&\sum_{i \in \mathcal{I}}\sum_{s \in \mathcal{K}_i^+}\mathbf{x}_{is} = \mathbf{d}, \\
		&(\widetilde{\mathbf{x}}_{i},\mathbf{y}_i) \in\Omega_i= 
		\left\{ (\widetilde{\mathbf{x}}_{i},\mathbf{y}_i)\left|
		\begin{aligned}
			&\mathbf{x}_{is} \in y_i^s\otimes \Lambda_{is}+(1-y^s_i)\otimes\Gamma_{is}, \\
			&\quad\quad\quad\quad\quad\quad\quad\quad\quad~s \in \mathcal{K}_i^+\\
			&\widetilde{\mathbf{x}}_{i} =  \left\{ \mathbf{x}_{is} , z_{is} \right\}_{s \in \mathcal{K}_i^+},~\mathbf{E}_{i}\mathbf{y}_{i}=\mathbf{\delta}_{i},\\
			&\mathbf{y}_{i}\in\{0,1\}^{k_i}
		\end{aligned}
		\right.\right\},\\
		&\quad\quad\quad\quad\quad\quad\quad\quad\quad\quad\quad\quad\quad\quad\quad\quad\quad\quad\quad\quad\quad\quad\quad\quad~i\in\mathcal{I},
	\end{aligned}
	\right.
\end{equation}
where $\Gamma_{is}:=\{\mathbf{0}\}$ for each $i\in\mathcal{I}$ and $s \in \mathcal{K}_i^+$. It is clear that (UC) is a special case of \textrm{(P0)}. Hence the perspective reformulation (P-per) for (P0) can be applied to (UC). We denote the perspective reformulation of (UC) as (UC-per), whose complete formulation is given in the Appendix I of Supplementary Materials. We remark that (UC-per) is exactly the dynamic programming-based formulation proposed in \cite{Bacci2023}.

For the cases involving some identical generating units, we partition $\mathcal{I}$ into $T$ subsets such that $\mathcal{I}=\cup_{t=1}^T\mathcal{I}_{t}$, each of which represents a set of identical generating units that have the same initial states.
Since for each $i\in\mathcal{I}$, $\mathbf{E}_i$ is totally unimodular, and $\mathbf{\delta}_i$ is an integer vector, the variable aggregation-based reformulation (P-agg) can be applied to (UC). We use (UC-agg) to represent the variable aggregation-based formulation of (UC), whose complete formulation is given in the Appendix I in Supplementary Materials.

Based on Theorem~\ref{thm4}, for formulation (UC-agg), not only can the symmetry be broken, but also the continuous relaxation lower bound can be very tight. These advantages may significantly improve the performance of a mixed-integer programming solver.

\section{Application to separable mixed-integer convex optimization problem}\label{sec6}

Another special case of \textrm{(P0)} is the mixed-integer separable convex programming problem
\begin{equation}\label{MISQP}
	\textrm{(SP)}~\left\{
	\begin{aligned}
		\min~& \sum_{i\in \mathcal{I}}\left(f_{i}(x_{i})+c_{i} y_{i}\right) \\
		\textrm{s.t.}~&\sum_{i\in \mathcal{I}} g_{il}(x_i) \leq d_l,~l \in \mathcal{L},\\
		& l_{i} y_{i} \leq x_{i} \leq u_{i} y_{i},~y_{i} \in\{0,1\}, ~i\in \mathcal{I},
	\end{aligned}
	\right.
\end{equation}
where $c_i,l_i,u_i\in\mathbb{R}$, all the functions $f_{i}(\cdot)$ and $g_{il}(\cdot)$ are univariate convex functions for $i\in\mathcal{I}$, and $d_l\in\mathbb{R}$ for $l\in\mathcal{L}$. We remark that several important problems in \citep{Agnetis2009,Frangioni2011} are all special cases of \textrm{(SP)}. Furthermore, if $c_i=0$ and $l_i=0$ for all $i\in\mathcal{I}$, then the problem can be reduced to the continuous separable optimization problem, which has been studied extensively in the literature \citep{Bretthauer2002,Edirisinghe2019,Schoot2022}. As pointed-out in \citep{Agnetis2012}, symmetry arises naturally in \textrm{(SP)} in some real-life applications such as the line cover problem. Based on the symmetric structure, we partition the set $\mathcal{I}$ into $\mathcal{I}=\bigcup_{t=1}^T \mathcal{I}_{t}$, where each $\mathcal{I}_{t}$ represents an equivalent class, and reformulate \textrm{(SP)} as follows:
\begin{equation}\label{MISQPR2}
	\left\{
	\begin{aligned}
		\min ~& \sum_{t\in\mathcal{T}}\sum_{i\in \mathcal{I}_t}\left(z_i+c_t y_{i}\right) \\
		\textrm{s.t.}~&\sum_{t\in\mathcal{T}}\sum_{i\in \mathcal{I}_t} w_{il} \leq d_{l},~l \in \mathcal{L},\\
		&f_t(x_{i})\leq z_i,~l_{t} y_{i} \leq x_{i} \leq u_{t} y_{i},~y_{i} \in\{0,1\},~i\in \mathcal{I}_t,~t\in\mathcal{T},\\
		&g_{tl}(x_i)\leq w_{il},~i\in \mathcal{I}_t,~t\in\mathcal{T},~l \in \mathcal{L}.
	\end{aligned}
	\right.
\end{equation}
For any $i\in \mathcal{I}$, we use $\widetilde{\mathbf{x}}_i:=\{x_i,\{w_{il}\}_{l\in\mathcal{L}},z_i\}$ to represent a vector with entries composed of $x_i$, $\{w_{il}\}_{l\in\mathcal{L}}$ and $z_i$.
For any $t\in\mathcal{T}$, we define the set $\Gamma_t:=\{\mathbf{0}\}$, and define $\Lambda_t$ as follows:
\begin{equation}
	\begin{aligned}
		\Lambda_t =\left\{\widetilde{\mathbf{x}}_i \left|\begin{array}{@{}ll}
			&g_{tl}(x_i)\leq w_{il},~ l \in \mathcal{L}\\
			&l_{t} \leq x_{i} \leq u_{t}\\
			& f_t(x_i)\leq z_i\\
			&\underline{z}_{t}\leq z_{i}\leq \bar{z}_{t}\\
			&\underline{w}_{tl}\leq w_{il}\leq \bar{w}_{tl},~ l \in \mathcal{L}
		\end{array}\right.\right\},
	\end{aligned}
\end{equation}
where $\underline{w}_{tl}=\min_{x_i\in [l_{t},u_{t}]}g_{tl}(x_i)$, $\bar{w}_{tl}=\max_{x_i\in [l_{t},u_{t}]}g_{tl}(x_i)$, $\underline{z}_{t}=\min_{x_i\in [l_{t},u_{t}]}f_{t}(x_i)$ and $\bar{z}_{t}=\max_{x_i\in [l_{t},u_{t}]}f_{t}(x_i)$.
Using the above notations, we can reformulate problem \eqref{MISQPR2} as follows:
\begin{equation}\label{MISQPPR1}
	\left\{
	\begin{aligned}
		\min  & \sum_{t\in\mathcal{T}}\sum_{i\in \mathcal{I}_t}\left(z_i+c_t y_{i}\right) \\
		\textrm{s.t.} &\sum_{t\in\mathcal{T}}\sum_{i\in \mathcal{I}_t} w_{il} \leq d_{l},~l \in \mathcal{L},\\
		&\widetilde{\mathbf{x}}_i:=\{x_i,\{w_{il}\}_{l\in\mathcal{L}},z_i\},~(\widetilde{\mathbf{x}}_i,y_i) \in \Omega_t,~i\in \mathcal{I}_t,~t\in\mathcal{T},
	\end{aligned}
	\right.
\end{equation}
where $\Omega_t := \left\{ (\widetilde{\mathbf{x}}_i,y_i)\left|\widetilde{\mathbf{x}}_i \in y_{i} \otimes \Lambda_{t} + (1-y_{i}) \otimes \Gamma_{t},~ y_{i} \in\{0,1\} \right.\right\}$.
It is easy to check that \eqref{MISQPPR1} is a special case of \textrm{(P0)}. Therefore, the reformulations \textrm{(P-per)} and \textrm{(P-agg)} for \textrm{(P0)} can be applied to \eqref{MISQPPR1}. We use (SP-per) and (SP-agg) to denote the reformulations of applying \textrm{(P-per)} and \textrm{(P-agg)} to the case of \eqref{MISQPPR1}, respectively, whose complete formulations are presented in the Appendix J of Supplementary Materials. The numerical performance of the three reformulations will be shown in the next section.

\section{Computational Experiments}\label{sec7}

In this section, we evaluate the computational performance of the proposed variable aggregation-based perspective reformulation technique. To this end, we conduct experiments on three classes of problems: (i) the unit commitment problem, (ii) the line cover problem, and (iii) general separable mixed-integer convex quadratic programming problems.

For each problem type, we compare multiple formulations to assess the effectiveness of the proposed method. Two key performance metrics are used for comparison:
\begin{itemize}
	\item[$\bullet$] Computational Efficiency - Measured by the total computational time and the number of branch-and-bound nodes explored during the solution process for each formulation. This metric reflects the effectiveness of the proposed reformulation in breaking problem symmetry and improving solver efficiency.
	\item[$\bullet$] Tightness of Continuous Relaxation - Evaluated based on the lower bound obtained from the continuous relaxation of each formulation and its relative gap to the global optimal value, defined as:
	$$\text{relative gap}:=\frac{|\text{OPT}-\text{LB}_{(\cdot)}|}{|\text{OPT}|},$$
	where OPT is the global optimal value, and LB$_{(\cdot)}$ is the lower bound provided by the continuous relaxation of the corresponding formulation.
\end{itemize}

All computational experiments are performed on a personal computer equipped with an Intel(R) Core(TM) Ultra 5 125H 3.60 GHz processor and 32GB RAM. The optimization problems are solved using Gurobi 12.0.0 through its Python interface. To ensure fair and consistent comparisons, the number of threads is fixed at 1 for all experiments, and a time limit of 1200 CPU seconds is imposed for each instance. These settings allow us to consistently compare the performance of different formulations and rigorously test the benefits of the proposed reformulation technique across diverse problem types.

\subsection{Unit commitment problem}
We first evaluate the proposed approach on the unit commitment problem, a fundamental problem in power systems where it is common to encounter multiple identical generating units. For this study, we employ Ostrowski's benchmark test set \citep{Knueven2017}, which is specifically designed to assess the solution efficiency of unit commitment problems with identical units. The test set contains 20 instances, each constructed by replicating some of the generating units originally described in \citep{Kazarlis1996}. Further details regarding the test instance parameters can be found in Section V of \citep{Knueven2017}.

To assess the effectiveness of the proposed variable aggregation-based perspective reformulation (UC-agg), we compare its performance against two formulations: (i) (UC-per): equivalent to the dynamic programming-based formulation proposed in \citep{Bacci2023}, and (ii) the classical 3-bin formulation \citep{Rajan2005}, which is widely regarded for its strong performance.

In our experiments, both the 3-bin and (UC-per) formulations follow the exact implementations described by \citep{Bacci2023}. It is important to emphasize that the proposed (UC-agg) formulation differs from the aggregation-based formulation introduced by \citep{Knueven2017}, as our formulation incorporates perspective reformulation techniques to directly handle the convex quadratic cost functions, rather than relying on piecewise linear approximations.

Each instance is solved using Gurobi with a MIPGap tolerance of $10^{-4}$, consistent with \citep{Bacci2023}. The computational results are presented in Table~\ref{tab1}, reporting the Problem ID (ID), Total number of generating units (Tol), Average number of identical units (Rep), Computational time in seconds (Time), Number of branch-and-bound nodes explored (Nodes). Moreover, a dash ``--'' indicates that the solver failed to solve the instance within the 1200 seconds time limit.

\begin{table}[bpht]
	\caption{Computational results on Ostrowski's instances}\label{tab1}
	\centering
	\scalebox{0.65}{
		\begin{tabular}{ccccccccc}
			\toprule
			\multirow{2}{*}{ID} & \multirow{2}{*}{Tol}  & \multirow{2}{*}{Rep}  &\multicolumn{2}{c}{3-bin}  & \multicolumn{2}{c}{\textrm{(UC-per)}} & \multicolumn{2}{c}{\textrm{(UC-agg)}}  \\
			\cmidrule(r){4-5}  \cmidrule(r){6-7} \cmidrule(r){8-9}
			&       &          & Time      & Nodes     & Time      & Nodes        & Time      & Nodes           \\
			\midrule
			1 & 28& 3.5 & 161.7680 & 27895.0  &-- &--& 0.9930 & 155.0 \\
			2 & 35& 4.4 & 22.8720 & 3818.0  & -- & --& 8.7300 & 1037.0 \\
			3 & 44& 5.5 & 16.3280 & 1240.0  & -- & -- & 4.1920 & 615.0 \\
			4 & 45& 5.6 & 17.9650 & 1882.0  & -- & -- & 6.9180 & 1604.0 \\
			5 & 49& 6.1 & 23.5750 & 1953.0  & -- & -- & 3.6530 & 818.0 \\
			6 & 50& 6.3 & 25.8940 & 1433.0 & -- & -- & 22.4200 & 2235.0 \\
			7 & 51& 6.4 & 35.5750 & 3422.0  & -- & -- & 7.9650 & 871.0 \\
			8 & 51& 6.4 & 28.5230 & 2016.0  & -- & -- & 11.1330 & 3113.0 \\
			9 & 52& 6.5 & 30.2800 & 2510.0  & --& -- & 19.6580 & 9610.0 \\
			10 & 54& 6.8 & 56.8100 & 3289.0  & -- & -- & 14.8550 & 2253.0 \\
			11 & 132& 16.5 & -- & --  & -- & -- & 8.2180 & 1270.0 \\
			12 & 156& 19.5 & 45.4530 &  876.0  & 631.3040 & 1.0& 17.7150 & 1046.0 \\
			13 & 156& 19.5 & 60.9420 & 844.0 & 650.4190 & 2023.0 & 35.8030 & 75.0 \\
			14 & 165& 20.6 & 37.7060 & 303.0  & 314.3900 & 1.0& 1.5860 & 31.0 \\
			15 & 167& 20.9 & 214.0810 & 5896.0  & -- & -- & 26.1850 & 3014.0 \\
			16 & 172& 21.5 & 24.2450 & 158.0  & 631.7160 & 1.0& 20.7230 & 1.0  \\
			17 & 182& 22.8 & 58.1690 & 629.0  & 954.1220 & 163.0& 0.7040 & 1.0  \\
			18 & 182& 22.8 & 163.9320 & 3418.0  & -- & -- & 8.9430 & 1322.0 \\
			19 & 183& 22.9 & 152.0440 & 3057.0  & -- &-- & 8.4260 & 1783.0 \\
			20 & 187& 23.4 & 119.3980 & 2251.0  & 674.6270 & 1.0& 0.9210 & 1.0\\
			\bottomrule
	\end{tabular}}
\end{table}

From Table~\ref{tab1}, it is evident that the proposed (UC-agg) formulation substantially outperforms the alternatives. Specifically: All instances are solved within 36 seconds using (UC-agg). In comparison, the 3-bin formulation requires over 100 seconds for five instances, with one instance unsolved within the time limit. Additionally, for the (UC-per) formulation, only six instances can be solved within 1200 seconds. This is attributed to its larger number of binary variables compared to the other formulations. These results highlight the effectiveness of (UC-agg) in both efficiently breaking symmetry and providing tight continuous relaxations.

To further evaluate relaxation tightness, we compare the lower bounds of the continuous relaxations, the optimal values, and their relative gaps for each formulation in Table~\ref{tab2}. It is observed that (UC-agg) and (UC-per) achieve essentially identical lower bounds (within numerical accuracy), both of which are significantly tighter than the 3-bin formulation. These numerical results corroborate our theoretical findings established in Theorem~\ref{thm4}, confirming the tightness of the proposed variable aggregation-based perspective reformulation.

\begin{table}[t]
	\caption{Different continuous relaxation-based lower bounds and relative gaps for Ostrowski's instances}\label{tab2}
	\centering%
	\scalebox{0.65}{
		\begin{tabular}{ccccc}
			\toprule
			ID      &3-bin &\textrm{(UC-per)}  &\textrm{(UC-agg)}  &Opt.val  \\
			\midrule
			1 &$3.8390\times 10^{6}(1.9559\times10^{-3})$ &$3.8437\times 10^{6}(7.1776\times10^{-4})$ &$3.8437\times 10^{6}(7.1777\times10^{-4})$ &$3.8465\times 10^{6}$ \\
			2 &$4.8265\times 10^{6}(1.9291\times10^{-3})$ &$4.8311\times 10^{6}(7.7341\times10^{-4})$ &$4.8311\times 10^{6}(7.7340\times10^{-4})$ &$4.8348\times 10^{6}$ \\
			3 &$5.1345\times 10^{6}(1.7396\times10^{-3})$ &$5.1424\times 10^{6}(2.1523\times10^{-4})$ &$5.1424\times 10^{6}(2.1523\times10^{-4})$ &$5.1435\times 10^{6}$ \\
			4 &$4.8028\times 10^{6}(1.4567\times10^{-3})$ &$4.8081\times 10^{6}(3.6100\times10^{-4})$ &$4.8081\times 10^{6}(3.6100\times10^{-4})$ &$4.8098\times 10^{6}$ \\
			5 &$5.4016\times 10^{6}(2.3316\times10^{-3})$ &$5.4137\times 10^{6}(1.0374\times10^{-4})$ &$5.4137\times 10^{6}(1.0376\times10^{-4})$ &$5.4142\times 10^{6}$ \\
			6 &$4.4132\times 10^{6}(4.2960\times10^{-3})$ &$4.4310\times 10^{6}(2.7486\times10^{-4})$ &$4.4310\times 10^{6}(2.7487\times10^{-4})$ &$4.4322\times 10^{6}$ \\
			7 &$5.8485\times 10^{6}(1.2884\times10^{-3})$ &$5.8546\times 10^{6}(2.3323\times10^{-4})$ &$5.8546\times 10^{6}(2.3326\times10^{-4})$ &$5.8560\times 10^{6}$ \\
			7 &$5.1801\times 10^{6}(2.6380\times10^{-3})$ &$5.1925\times 10^{6}(2.3667\times10^{-4})$ &$5.1925\times 10^{6}(2.3667\times10^{-4})$ &$5.1938\times 10^{6}$ \\
			9 &$5.6293\times 10^{6}(2.5132\times10^{-3})$ &$5.6425\times 10^{6}(1.6992\times10^{-4})$ &$5.6425\times 10^{6}(1.6991\times10^{-4})$ &$5.6435\times 10^{6}$ \\
			10 &$5.0743\times 10^{6}(4.3465\times10^{-3})$ &$5.0952\times 10^{6}(2.5269\times10^{-4})$ &$5.0952\times 10^{6}(2.5269\times10^{-4})$ &$5.0965\times 10^{6}$ \\
			11 &$1.5865\times 10^{7}(1.3122\times10^{-3})$ &$1.5883\times 10^{7}(2.1014\times10^{-4})$ &$1.5883\times 10^{7}(2.0992\times10^{-4})$ &$1.5886\times 10^{7}$ \\
			12 &$1.7253\times 10^{7}(1.3102\times10^{-3})$ &$1.7274\times 10^{7}(1.0274\times10^{-4})$ &$1.7274\times 10^{7}(1.0278\times10^{-4})$ &$1.7276\times 10^{7}$ \\
			13 &$1.6927\times 10^{7}(1.7760\times10^{-3})$ &$1.6954\times 10^{7}(1.4674\times10^{-4})$ &$1.6954\times 10^{7}(1.4671\times10^{-4})$ &$1.6957\times 10^{7} $ \\
			14 &$2.0200\times 10^{7}(1.6584\times10^{-3})$ &$2.0233\times 10^{7}(3.0053\times10^{-5})$ &$2.0233\times 10^{7}(3.0057\times10^{-5})$ &$2.0234\times 10^{7}$ \\
			15 &$1.7375\times 10^{7}(2.7703\times10^{-3})$ &$1.7421\times 10^{7}(1.3432\times10^{-4})$ &$1.7421\times 10^{7}(1.3433\times10^{-4})$ &$1.7424\times 10^{7} $ \\
			16 &$1.9552\times 10^{7}(1.1183\times10^{-3})$ &$1.9574\times 10^{7}(8.4627\times10^{-5})$ &$1.9574\times 10^{7}(8.4628\times10^{-5})$ &$1.9575\times 10^{7}$ \\
			17 &$1.9712\times 10^{7}(2.2026\times10^{-3})$ &$1.9755\times 10^{7}(4.4694\times10^{-5})$ &$1.9755\times 10^{7}(4.4712\times10^{-5})$ &$1.9756\times 10^{7} $ \\
			18 &$1.9623\times 10^{7}(2.0091\times10^{-3})$ &$1.9661\times 10^{7}(1.0774\times10^{-4})$ &$1.9661\times 10^{7}(1.0774\times10^{-4})$ &$1.9663\times 10^{7} $ \\
			19 &$2.0159\times 10^{7}(1.5809\times10^{-3})$ &$2.0188\times 10^{7}(1.4108\times10^{-4})$ &$2.0188\times 10^{7}(1.4108\times10^{-4})$ &$2.0191\times 10^{7} $ \\
			20 &$1.9742\times 10^{7}(1.9517\times10^{-3})$ &$1.9779\times 10^{7}(7.3945\times10^{-5})$ &$1.9779\times 10^{7}(7.3954\times10^{-5})$ &$1.9781\times 10^{7} $ \\
			\bottomrule
	\end{tabular}}
	\footnotesize{
		\begin{tablenotes}%
		\item [1] Note: The numbers outside and in parentheses are lower bounds and relative gaps, respectively.
	\end{tablenotes} }
\end{table}

\subsection{Line cover problem}\label{sec72}

Next, we examine the line cover problem, introduced by \citep{Agnetis2009} and subsequently studied in \citep{Frangioni2011,Agnetis2012}. The problem involves selecting a subset of sensors and determining their locations to cover a line segment of length $d$ (e.g., a road, river, or pipeline). If a sensor $i\in\mathcal{I}$ covers a portion of the line of length  $x_i$, the associated cost is given by $a_i x_i^2+c_i$, where $c_i$ represents a setup cost and $a_i x_i^2$ is a quadratic operational cost. The objective is to select sensors and allocate coverage lengths to minimize the total cost.
A mixed-integer quadratic programming formulation is employed as
\begin{equation}\label{lcProb}
	\textrm{(LC)}~\left\{
	\begin{aligned}
		\min & \sum_{i \in \mathcal{I}}\left(a_i x_i^2+c_i y_i\right) \\
		\mbox{s.t.}~  &\sum_{i \in \mathcal{I}} x_i = d,\\
		& 0 \leq x_{i} \leq u_i y_i,~y_i \in\{0,1\},~ i \in \mathcal{I},
	\end{aligned}
	\right.
\end{equation}
where $y_i$ is a binary variable indicating if the sensor is selected, and $u_i>0$ denotes the maximum radius that can be covered by sensor $i$.

As (LC) constitutes a special case of separable convex quadratic optimization, both the perspective reformulation (SP-per) and the proposed variable aggregation-based reformulation (SP-agg) are applicable. We denote the resulting formulations as (LC-per) and (LC-agg), respectively.

Computational test instances are generated following the procedures described in \citep{Frangioni2011}. Specifically: (i) each $a_i$ is uniformly sampled from $[n,C_{\max}]$, where $n$ is the total number of sensors and $C_{\max}$ is randomly selected from $\{10n,20n,30n\}$, (ii) each $c_i$ is a random integer in $\{1,\ldots,n\}$, and (iii) all upper bounds $u_i$ and the line length $d$ are set to 1. In our experiments, $n=18000$ sensors are considered. To introduce symmetry, parameters of the first $T$ sensors are generated and then replicated $N$ times, with $T=n/N$. We test five settings of $N\in\{10,15,20,30,50\}$, generating five random instances per setting.

All instances are solved using Gurobi, with the MIPGap tolerance set to $10^{-6}$ (as continuous relaxation gaps for (LC-per) and (LC-agg) are already below $10^{-4}$). Additionally, symmetry-breaking inequalities are introduced for the binary variables associated with each group of identical sensors, enforcing $y_1 \geq y_2 \geq \ldots \geq y_{N_t}$ for each set $\mathcal{I}_t:=\{1,\ldots,N_t\}$. These enhancements are applied to both (LC) and (LC-per), with results compared to (LC-agg). The computational outcomes are summarized in Table~\ref{tab3}, where each row reports the average computational time and average number of branch-and-bound nodes over five randomly generated instances.

From Table~\ref{tab3}, it is evident that both (LC-per) and (LC-agg) are considerably more efficient than (LC), confirming the effectiveness of perspective reformulation in reducing the relaxation gap. Moreover, (LC-agg) consistently outperforms (LC-per), demonstrating the additional benefit of symmetry breaking through aggregation.

To assess relaxation tightness, we record the average lower bound, average optimal value, and average relative gaps for each formulation in Table~\ref{tab4}. The results indicate that (LC-agg) is as tight as (LC-per), and both produce significantly tighter lower bounds compared to (LC), further validating the advantage of the proposed approach.

\begin{table}[h]
	\caption{Computational results for randomly generated test instances of line cover problem}\label{tab3}
	\centering
	\scalebox{0.7}{
		\begin{tabular}{ccccccc}
			\toprule
			\multirow{2}{*}{$(T,N)$} & \multicolumn{2}{c}{\textrm{(LC)}}  & \multicolumn{2}{c}{\textrm{(LC-per)}}   & \multicolumn{2}{c}{\textrm{(LC-agg)}} \\
			\cmidrule(r){2-3}  \cmidrule(r){4-5} \cmidrule(r){6-7}
			& Time  & Nodes  & Time  & Nodes   & Time  & Nodes   \\
			\midrule
			(1800,10)     & 48.6778    & 2.6       & 6.3282   & 2.6       & 0.6100  & 1.4   \\
			(1200,15)    & 48.0874    & 4.6        & 4.8298   & 5.2        & 0.2598  & 2.6    \\
			(900,20)     & 53.5290     & 8.6       & 4.9538   & 6.2     & 0.2112  & 2.2    \\
			(600,30)     & 68.7288     & 4.2         & 5.6392   & 9.8      & 0.2182  & 1.8  \\
			(360,50)   & 92.3584     & 6.8        & 6.2186   & 18.4       & 0.1606  & 3.4  \\
			\bottomrule
	\end{tabular}}
\end{table}

\begin{table}[h]
	\caption{Different average lower bounds and average relative gaps for instances of line cover problem}\label{tab4}
	\centering
	\scalebox{0.7}{
		\begin{tabular}{ccccccc}
			\toprule
			$(T,N)$   &\textrm{(LC)}               &\textrm{(LC-per)}            &\textrm{(LC-agg)}               &Opt.val  \\
			\midrule
			(1800,10) &$472.20 (8.59\times10^{-1})$   &$3389.77 (1.46\times10^{-4})$   &$3389.77(1.46\times10^{-4})$   &$3390.33$     \\
			(1200,15) &$482.22(8.73\times10^{-1})$    &$3859.27(1.52\times10^{-4})$    &$3859.27(1.52\times10^{-4})$   &$3859.97$     \\
			(900,20)  &$455.73(8.78\times10^{-1})$    &$3949.70(3.78\times10^{-5})$    &$3949.70(3.78\times10^{-5})$   &$3949.87$     \\
			(600,30)  &$439.07(8.59\times10^{-1})$    &$3423.62(7.31\times10^{-5})$    &$3423.62(7.31\times10^{-5})$  &$3424.01$     \\
			(360,50)  &$502.20(9.06\times10^{-1})$    &$5699.58(2.59\times10^{-4})$    &$5699.58(2.59\times10^{-4})$   &$5701.62$     \\
			\bottomrule
	\end{tabular}}
\end{table}

\subsection{Separable mixed-integer convex quadratic optimization problem}


We further evaluate the proposed methodology on a separable mixed-integer convex quadratic optimization problem, a special case of (SP). The problem is formulated as follows:
\begin{equation}\label{MIQP}
	\textrm{(SQP)}~\left\{
	\begin{aligned}
		\min  & \sum_{i\in \mathcal{I}}\left(a_i x_i^2+b_i x_i+c_i y_i\right) \\
		\textrm{s.t.} &\sum_{i\in\mathcal{I}}(a_{li} x_i^2+b_{li} x_i)\leq d_l,~l=1,\ldots,m-1,\\
		&\sum_{i\in\mathcal{I}}x_i=d_m,\\
		& l_{i} y_{i} \leq x_{i} \leq u_{i} y_{i},~y_{i} \in\{0,1\}, ~i\in \mathcal{I},
	\end{aligned}
	\right.
\end{equation}
where $a_i,b_i\geq0$, $l_i,u_i,c_i\in\mathbb{R}$ for $i \in \mathcal{I}$, $d_l\in \mathbb{R}$ for $l=1,\ldots,m$, and $a_{li},b_{li}\geq 0$ for $i\in\mathcal{I}$ and $l=1,\ldots,m$. This formulation consists of $m-1$ separable convex quadratic constraints and one linear equality constraint, which is added to prevent from generating trivial test instances.

Both reformulations introduced in Section 6, including (SP-per) and (SP-agg), are applicable to (SQP), denoted here as (SQP-per) and (SQP-agg), respectively. For the baseline (SQP) and (SQP-per), symmetry breaking inequalities of the form $y_1 \geq y_2 \geq \ldots \geq y_{N_t}$ were incorporated for each set  $\mathcal{I}_t:=\{1,\ldots,N_t\}$ to enhance computational performance.


Test instances were generated following the distribution described in Section 6.2 of \citep{Luo2023}, modified to include binary variables and multiple separable convex quadratic constraints with (i) for each $i$, parameters $a_i$ and $c_i$ were sampled from $[0,1]$, and $b_i$ from $[2,5]$, (ii) variable bounds were set to $[l_i,u_i]=[-1,1]$, (iii) a feasible solution $\hat{\mathbf{x}}$ was randomly generated, with each $\hat{x}_i$ drawn from $[-1,1]$, (iv)  for each $l\in\{1,\ldots,m-1\}$, parameters $a_{li}$ and $b_{li}$ were sampled from $[0,2]$ and $[0,5]$, respectively, with the corresponding $d_l$ set to ensure feasibility via $d_l=\sum_{i\in\mathcal{I}}(a_{li} \hat{x}_i^2+b_{li}\hat{x}_i)$, and (v) $d_m$ was set to $\sum_{i\in\mathcal{I}}\hat{x}_i$.

As in Section~\ref{sec72}, the total number of variables was set to $n=18000$, with $m=4$, $N\in\{10,15,20,30,50\}$, and $T=n/N$. The parameters $(a_i,b_i)$ were randomly generated for the first $T$ elements, then replicated $N$ times to introduce symmetric structure. For each setting, five random instances were generated and solved using Gurobi, with solver settings consistent with Section~\ref{sec72}.

Table~\ref{tab5} summarizes computational time and enumerated nodes. It clearly demonstrates that (SQP-agg) significantly outperforms the other two formulations, achieving substantial reductions in both computational time and the number of branch-and-bound nodes. These results further confirm the effectiveness of the proposed reformulation in handling symmetry and improving solution efficiency.

Table~\ref{tab6} reports the average lower bound and relative gap for each formulation. It shows that the continuous relaxations of (SQP-agg) and (SQP-per) achieve equivalent lower bounds (within numerical tolerance), both substantially tighter than those obtained from the original (SQP) formulation. This result reinforces the theoretical findings from Section 6.

Across all three classes of problems, the proposed variable aggregation-based perspective reformulations consistently outperformed classical perspective reformulations without variable aggregation, especially on large symmetric instances.

\begin{table}[h]
	\caption{Computational results for mixed-integer separable quadratic optimization problems}\label{tab5}
	\centering
	\scalebox{0.7}{
		\begin{tabular}{ccccccccc}
			\toprule
			\multirow{2}{*}{$(T,N)$}  & \multicolumn{2}{c}{\textrm{(SQP)}}  & \multicolumn{2}{c}{\textrm{(SQP-per)}} & \multicolumn{2}{c}{\textrm{(SQP-agg)}} \\
			\cmidrule(r){2-3}  \cmidrule(r){4-5} \cmidrule(r){6-7}
			& Time  & Nodes    & Time  & Nodes   & Time & Nodes\\
			\midrule
			$(1800,10)$  & 124.8916  & 9.0    & 17.8260  & 3.0    & 1.2414  & 1.0   \\
			$(1200,15)$ & 120.0990 & 13.2   & 21.6206  & 99.6   & 0.6068  & 1.0   \\
			$(900,20)$  & 153.0214  & 43.0     & 30.0326  & 17.8    & 0.5786  & 1.0  \\
			$(600,30)$  & 174.5346  & 94.0     & 56.5870  & 49.8  & 0.3554  & 1.0   \\
			$(360,50)$  & 313.3260  & 93.0     & 65.3550  & 31.6   & 0.4192  & 5.8   \\
			\bottomrule
	\end{tabular}}
\end{table}

\begin{table}[h]
	\caption{Different average lower bounds and average relative gaps for mixed-integer separable quadratic optimization problems}\label{tab6}
	\centering
	\scalebox{0.7}{
		\begin{tabular}{cccccc}
			\toprule
			$(T,N)$   &\textrm{(SQP)}  &\textrm{(SQP-per)} &\textrm{(SQP-agg)} &Opt.val    \\ \midrule
			(1800,10)  &$-2717.94(1.81\times10^{-1})$         & $-2341.67(2.12\times10^{-7})$    & $-2341.67(2.53\times10^{-7})$    & $-2341.67$    \\
			(1200,15)  &$-2548.69(2.24\times10^{-1})$         & $-2145.47(2.54\times10^{-7})$    & $-2145.47(2.21\times10^{-7})$    & $-2145.47$    \\
			(900,20)   &$-2256.39(3.20\times10^{-1})$         & $-1853.64(3.48\times10^{-7})$    & $-1853.64(1.17\times10^{-7})$    & $-1853.64$    \\
			(600,30)   &$-3050.22(8.11\times10^{-1})$         & $-2645.61(5.03\times10^{-7})$    & $-2645.61(1.12\times10^{-7})$    & $-2645.61$    \\
			(360,50)   &$-2863.20(6.52\times10^{-1})$         & $-2431.03(1.80\times10^{-6})$    & $-2431.03(2.00\times10^{-6})$    & $-2431.03$    \\
			\bottomrule
	\end{tabular}}
\end{table}

\section{Conclusion}\label{sec8}
This paper proposes a novel variable aggregation-based perspective reformulation for addressing symmetry in mixed-integer convex optimization problems. The key innovation is the introduction of perspective functions to precisely characterize the feasible region of aggregated variables. By integrating perspective reformulation techniques with variable aggregation, we develop a powerful new formulation that not only handles symmetry but also tightens the continuous relaxation, yielding high-quality lower bounds.

Our main theoretical contribution is the exact convex hull characterization for the feasible region of each set of aggregated variables under the proposed formulation. We prove that the continuous relaxation of the aggregated formulation recovers the convex hull, ensuring tightness and enhancing the theoretical foundation for variable aggregation-based reformulations. Comparative analysis of lower bounds across different reformulations further confirms the superior performance achieved by our method.

Extensive computational experiments validate the practical effectiveness of the proposed reformulation technique. The results demonstrate significant improvements in both the quality of continuous relaxation bounds and the ability to break problem symmetry. These combined benefits lead to notable performance gains in branch-and-bound algorithms, as evidenced by substantial improvements in solver efficiency when using Gurobi.

Overall, this work advances the state-of-the-art in mixed-integer convex optimization by providing a theoretically sound and computationally effective approach for tackling symmetry. The proposed reformulation has the potential to improve solution methodologies for a wide range of real-world optimization problems characterized by symmetric structures.

\section*{Acknowledgements}
This research was supported by National Natural Science Foundation of China Grant No. 12171151, Grant No. T2293774, and by the Fundamental Research Funds for the Central Universities E2ET0808X2.
%
%


\bibliographystyle{elsarticle-harv} 
\bibliography{refs.bib}



%
%
%
\end{document}